\documentclass[11pt]{amsart}

\usepackage{accents}
\usepackage[all]{xy}
\usepackage{amsfonts}
\usepackage{amsmath}
\usepackage{amssymb}
\usepackage{amsthm}
\usepackage{color}
\usepackage{epstopdf}
\usepackage{fancyhdr}
\usepackage{float}
\usepackage{graphicx}
\usepackage{hyperref}
\usepackage{latexsym}
\usepackage{mathrsfs}
\usepackage[numbers]{natbib}
\usepackage{sseq}
\usepackage{tikz-cd}
\usetikzlibrary{decorations.pathmorphing,decorations.pathreplacing,patterns}
\usepackage{url}
\usepackage{verbatim}
\usepackage{stmaryrd}

\usepackage[full]{textcomp}
\usepackage[sups]{Baskervaldx}
\usepackage{cabin}
\usepackage[varqu,varl]{inconsolata}
\usepackage[baskervaldx,bigdelims,vvarbb]{newtxmath}
\usepackage[cal=cm]{mathalfa}

\usepackage{marginnote}

\theoremstyle{plain}
\newtheorem{theorem}{Theorem}[section]

\newtheorem{corollary}[theorem]{Corollary}
\newtheorem{lemma}[theorem]{Lemma}
\newtheorem{proposition}[theorem]{Proposition}

\theoremstyle{remark}
\newtheorem{remark}[theorem]{Remark}

\theoremstyle{definition}

\newtheorem*{aside*}{Aside}

\newtheorem{example}[theorem]{Example}

\newtheorem{proposition-definition}[theorem]{Proposition-Definition}

\newtheorem*{acknowledgements}{Acknowledgements}
\newtheorem*{userguide}{User's guide}
\newtheorem*{futuredirections}{Future directions}

\DeclareMathOperator{\ev}{ev}
\DeclareMathOperator{\vdim}{vdim}

\newcommand{\diff}{\operatorname{d}}

\newcommand{\C}{\mathbb{C}}
\newcommand{\T}{\operatorname{T}}

\newcommand{\PP}{\mathbb{P}}
\newcommand{\OO}{\mathcal{O}}

\newcommand{\tCoh}[1]{{\mathbf{t}}(#1)}
\newcommand{\qCoh}[1]{{\mathbf{q}}(#1)}
\newcommand{\pCoh}[1]{{\mathbf{p}}(#1)}

\newcommand{\HH}{\operatorname{H}}
\newcommand{\cohid}{\mathbb{1}_X}

\newcommand{\M}[4]{\overline{\mathcal{M}}_{#1,#2}(#3,#4)}

\newcommand{\virt}[1]{[#1]^{\operatorname{virt}}}

\newcommand{\bcd}{\begin{center}\begin{tikzcd}}
\newcommand{\ecd}{\end{tikzcd}\end{center}}

\newcommand{\pt}{\text{pt}}

\newcommand{\Cst}{\mathbb{C}^*}
\newcommand{\Lcal}{\mathcal{L}}
\newcommand{\Hcal}{\mathcal{H}}
\newcommand{\Fcal}{\mathcal{F}}

\begin{document}
 
\title[Fundamental solution and relative stable maps]{The fundamental solution matrix and relative~stable maps}
\author{Navid Nabijou}

\begin{abstract} Givental's Lagrangian cone $\Lcal_X$ is a Lagrangian submanifold of a symplectic vector space which encodes the genus-zero Gromov--Witten invariants of $X$. Building on work of Braverman, Coates has obtained the Lagrangian cone as the push-forward of a certain class on the moduli space of stable maps to $X \times \PP^1$. This provides a conceptual description for an otherwise mysterious change of variables called the dilaton shift.

In this article we recast this construction in its natural context, namely the moduli space of stable maps to $X \times \PP^1$ relative the divisor $X \times \infty$. We find that the resulting push-forward is another familiar object, namely the transform of the Lagrangian cone under the action of the fundamental solution matrix. This hints at a generalisation of Givental's quantisation formalism to the setting of relative invariants. Finally, we use a hidden polynomiality property implied by our construction to obtain a sequence of universal relations for the Gromov--Witten invariants, as well as new proofs of several foundational results concerning both the Lagrangian cone and the fundamental solution matrix. \end{abstract}
\maketitle
\setcounter{tocdepth}{1}
\tableofcontents

\section{Introduction}
The Gromov--Witten invariants of a smooth projective variety $X$ are defined as certain intersection numbers on moduli spaces of stable maps to $X$. They can be thought of as counting curves of specified genus and degree passing through specified cycles in $X$. Their intrinsic interest aside, Gromov--Witten invariants have connections to numerous other areas of mathematics, from representation theory to symplectic topology. In algebraic geometry they have been used in the proofs of classification theorems, as a tool for distinguishing non-deformation-equivalent varieties.

Many results in Gromov--Witten theory are expressed most cleanly via generating functions, that is, formal functions (usually polynomials or power series) whose coefficients are given by Gromov--Witten invariants. Oftentimes, a simple identity involving generating functions is all that is needed to express a relationship which, on the level of indvidual invariants, is extremely complicated. There is an underlying reason for this: Gromov--Witten theory has deep connections to theoretical physics, through which the aforementioned generating functions appear as the ``partition functions'' of physical theories. This circle of ideas has been extremely influential for the development of the subject, with the first major result in this direction being the celebrated Mirror Theorem \cite{CandelasEtAl, GiventalEquivariant, GiventalMirrorToric}.

In keeping with this spirit, A. Givental describes in \cite{GiventalQuantisation} a \emph{quantisation formalism} for Gromov--Witten invariants. In the genus-zero setting (when no ``quantisation'' is actually required), this amounts to encoding the Gromov--Witten invariants of $X$ in a \emph{Lagrangian cone}
\begin{equation*} \Lcal_X \subseteq \Hcal \end{equation*}
inside a certain symplectic vector space $\Hcal$, now called the \emph{Givental space}. The data of the cone $\Lcal_X$ is equivalent to the data of the generating functions discussed earlier, but it turns out to be a good idea to treat $\Lcal_X$ as a geometric object in its own right; many statements in Gromov--Witten theory can then be translated into statements about how $\Lcal_X$ transforms under certain symplectomorphisms of $\Hcal$.

The benefits of this quantisation formalism are twofold. From a theoretical viewpoint, it can be used to make rigorous sense of a number of deep predictions coming from physics. On the other hand, from a practical point of view, it has proven to be an extremely versatile framework in which to formulate and prove statements about Gromov--Witten invariants. Indeed, there are many results in Gromov--Witten theory which would be difficult to even state without the quantisation formalism: examples include the quantum Riemann--Roch formula \cite{CoatesGivental}, the crepant transformation conjecture \cite{CoatesIritaniJiangCrepant}, the Virasoro conjecture and various versions of the ``genus zero implies higher genus'' principle \cite{GiventalSemisimple}.

Building on work of Braverman \cite{Braverman}, T. Coates shows in \cite{CoatesLagrangianConeS1} that $\Lcal_X$ can be obtained as a ($\C^*$-localised) push-forward from the moduli space of stable maps to $X \times \PP^1$ (usually called the \emph{graph space}). This is motivated by Givental's heuristic description of $\Hcal$ as the $S^1$-equivariant cohomology of the loop space of $X$ \cite{GiventalHomological}, and gives a natural geometric interpretation for a mysterious change of variables, called the ``dilaton shift'', which is essential to the quantisation formalism. 

Coates' construction requires restricting to a certain open substack of the moduli space of stable maps to $X \times \PP^1$, before localising to a proper fixed locus (with respect to the natural $\C^*$-action on the moduli space) in order to push forward. With hindsight, this is really the push-forward from \emph{one} of the $\Cst$-fixed loci in the moduli space of \emph{relative} stable maps to the pair $(X \times \PP^1, X \times \infty)$.

A natural question to ask is then: what happens if we sum over \emph{all} the fixed loci? In this article we provide the answer (see Proposition~\ref{main result}): the result is the transform of the Lagrangian cone under the action of the fundamental solution matrix. The main tools used in the proof are the relative virtual localisation formula \cite[Theorem~3.6]{GraberVakil}, a virtual push-forward theorem for relative stable maps to the non-rigid target \cite[Theorem~5.2.7]{GathmannThesis} and a comparison lemma for psi classes, which we prove in \S \ref{subsection Lemma on psi classes}.

Because we are now summing over all fixed loci, we know that the resulting class must actually belong to the \emph{non-localised} equivariant cohomology. In practice, this means the following: we push-forward and obtain a class which, a priori, looks like a rational function in $z$; however we know that, after performing suitable cancellations, we must end up with a polynomial (here $z$ denotes the $\C^*$-equivariant parameter). We use this observation to give new and simple proofs of a number of foundational results belonging to the quantisation formalism theory.

\begin{futuredirections} This construction provides a hint as to how one might obtain a quantisation formalism for relative (or logarithmic) Gromov--Witten invariants; see Remark \ref{remark Log}. This was in fact the original motivation for this work.\end{futuredirections}

\begin{userguide} Readers familiar with Gromov--Witten theory and the quantisation formalism may skip straight to \S \ref{subsection Statement of the proposition} where we give the statement of the main result.  For the uninitiated, we provide in \S\S \ref{subsection Gromov--Witten invariants}-\ref{subsection Relative stable maps} a brief introduction to Gromov--Witten invariants, the Lagrangian cone and relative Gromov--Witten theory. The proof of the main result is given in \S \ref{section Proof of the proposition}; this is mostly a computation, with the only geometric content being a lemma on psi classes which we prove in \S \ref{subsection Lemma on psi classes}. Finally in \S \ref{subsection Variants and applications} we provide examples of how the ``hidden polynomiality'' implied by our construction can be used to obtain universal relations for the Gromov--Witten invariants, as well as new proofs of a number of standard results concerning the Lagrangian cone and the fundamental solution matrix.
\end{userguide}

\begin{acknowledgements} I owe a great deal of thanks to Tom Coates, for first suggesting this project, for patiently explaining the quantisation formalism to me and for pointing out some of the applications presented in \S \ref{subsection Variants and applications}. I would also like to thank Pierrick Bousseau, Elana Kalashnikov and Mark Shoemaker for useful discussions, and the referee for a number of helpful comments.

The author is supported by an EPSRC Standard DTP Scholarship and by the Engineering and Physical Sciences Research Council grant EP/L015234/1: the EPSRC Centre for Doctoral Training in Geometry and Number Theory at the Interface.\end{acknowledgements}

\section{Background and statement of the main result}
\subsection{Gromov--Witten invariants} \label{subsection Gromov--Witten invariants}
Throughout we fix a smooth projective variety $X$ over the complex numbers. The genus-zero Gromov--Witten invariants of $X$ are defined as certain integrals over moduli spaces of stable maps to $X$ \cite{KontsevichEnumeration}. Fixing a number $n \geq 0$ of marked points and a curve class $\beta \in \HH_2^+(X)$ (where $\HH_2^+(X) \subseteq \HH_2(X)$ is the submonoid of \emph{effective classes}, i.e. those which can be represented by algebraic curves), the moduli space of stable maps
\begin{equation*} \M{0}{n}{X}{\beta} \end{equation*}
parametrises holomorphic maps $f\colon C \to X$ of class $\beta$, where $C$ is a nodal curve of arithmetic genus zero with $n$ distinct non-singular marked points. There is a stability condition which stipulates that $f$ can only have finitely many automorphisms; this is equivalent to requiring that every component of $C$ which is contracted by $f$ contains at least $3$ special points (either marked points or nodes). The resulting moduli space is a proper Deligne--Mumford stack, with virtual dimension (sometimes also referred to as the expected dimension):
\begin{equation*} \vdim \M{0}{n}{X}{\beta} = \dim X - 3 -K_X\cdot \beta + n.\end{equation*}
Although it is not in general smooth or even irreducible, and can contain components in excess of the virtual dimension, it admits a virtual fundamental class of pure dimension equal to the virtual dimension: this should be thought of as the fundamental class of some suitably generic perturbation of the moduli space. The Gromov--Witten invariants are then defined as:
\begin{equation*} \langle \gamma_1 \psi_1^{k_1},\ldots,\gamma_n \psi_n^{k_n} \rangle^X_{0,n,\beta}:= \int_{\virt{\M{0}{n}{X}{\beta}}} \prod_{i=1}^n \ev_i^*(\gamma_i) \cdot \psi_i^{k_i}.\end{equation*}
In the above formula each $\gamma_i \in \HH^*(X)$ is a class on the target, while each $\psi_i$ is a class on the moduli space itself which has to do with the complex structure of the source curve near the $i$th marked point. Ignoring these latter terms (whose geometric interpretation is somewhat more involved \cite{GraberKockPandharipande}) the Gromov--Witten invariant $\langle \gamma_1,\ldots,\gamma_n \rangle^X_{0,n,\beta}$ should be thought of as a ``virtual'' count of rational curves in $X$ of class $\beta$ which pass through (representatives of) the classes $\gamma_1,\ldots,\gamma_n$.  For a more detailed discussion of stable maps and Gromov--Witten invariants, see \cite{FultonPandharipande}, \cite[\S 7]{CoxKatz}, \cite[\S 1]{GathmannThesis}.

\subsection{Givental space} \label{subsection Givental space} The Lagrangian cone $\Lcal_X$ is a geometric object which encodes all the genus-zero Gromov--Witten invariants of $X$. It can be viewed as the graph of a certain generating function for these invariants. This generating function must keep track, through its formal variables, of both the cohomological insertions $\gamma_i$ and the exponents $k_i$ of the classes $\psi_i$. We begin by defining a vector space $\Hcal$ whose co-ordinates will give precisely these formal variables; the Lagrangian cone will then be a submanifold of $\Hcal$.

We set $\HH^*(X) = \HH^*(X; \Lambda)$ where $\Lambda$ is some (unspecified) field of characteristic zero; for the moment it is safe to take $\Lambda=\C$, but later we will need to consider larger fields. We assume (for notational simplicity) that $X$ has only even cohomology, and choose a homogeneous basis $\varphi_0, \ldots, \varphi_N$ such that $\varphi_0 = \cohid$ is the unit element. We let $\varphi^0,\ldots,\varphi^N$ denote the dual basis with respect to the Poincar\'e pairing $(\cdot\, ,\cdot)$, so that:
\begin{equation*} (\varphi_\alpha,\varphi^\beta) = \delta_{\alpha}^{\beta}.\end{equation*}
The \emph{Givental space} $\Hcal$ is a certain infinite-dimensional symplectic vector space (over $\Lambda$) associated to $X$. It is defined as the space of formal Laurent series in a single variable $z^{-1}$ with coefficients in $\HH^*(X)$:
\begin{equation*} \Hcal := \HH^*(X) [z, z^{-1} \rrbracket = \left\{ \sum_{-\infty \leq k \leq m} q_k z^k \colon q_k \in \HH^*(X) \right\}.\end{equation*}
The notation above is meant to indicate that each series has only finitely many positive powers of $z$, but can have infinitely many negative powers. The powers of $z^{-1}$ will keep track of the exponents of the psi classes.

There is a symplectic form $\Omega$ on $\Hcal$ defined as follows
\begin{align*} \Omega \colon & \mathcal{H} \times \mathcal{H} \to \Lambda \\
& ( f(z), g(z) ) \mapsto \operatorname{Res}_{z=0} ( f(-z),g(z)) \operatorname{d}\!z \end{align*}
where $( f(-z),g(z))$ is the Poincar\'e pairing (extended linearly from $\HH^*(X)$ to $\Hcal$), and $\operatorname{Res}_{z=0}$ simply means that we take the coefficient of $z^{-1}$ in the resulting Laurent series. A straightforward computation verifies that $\Omega$ is indeed a symplectic form.

\begin{example} Take $X=\pt$ so that $\HH^*(X)=\Lambda$. Then $\Hcal = \Lambda[z,z^{-1}\rrbracket$ and $\Omega$ is given by:
\begin{equation*} \Omega \left( \sum_{k} a_k z^k , \sum_l b_l z^l \right) = \operatorname{Res}_{z=0} \left( \sum_k \sum_l (-1)^k a_k b_l z^{k+l} \right) = \sum_{k+l=-1} (-1)^k a_k b_l.\end{equation*}
Notice that this sum is finite since the terms which appear must have either $k$ or $l$ non-negative, and there are only finitely many such values for which $a_k$ and $b_l$ are both non-zero.
\end{example}

Thus $(\Hcal,\Omega)$ is an infinite-dimensional symplectic vector space. We will now write down Darboux co-ordinates. It is clear that the following defines a basis for $\mathcal{H}$:
\begin{align*} A_{\alpha}^k & := \varphi_{\alpha} z^k \qquad\qquad\ \, k \geq 0, \alpha=0,\ldots,N\\
B^{\gamma}_l & := \varphi^{\gamma} (-z)^{-1-l} \qquad l \geq 0, \gamma=0,\ldots,N.
\end{align*}
It is also easy to see that these give Darboux co-ordinates, i.e. that we have:
\begin{align*} & \Omega(A_{\alpha}^k, A_{\alpha^\prime}^{k^\prime}) = 0, \qquad \Omega(B^{\gamma}_l,B^{\gamma^\prime}_{l^\prime}) = 0, \qquad\Omega(A_{\alpha}^k, B^{\gamma}_l) = -\delta_{\alpha}^{\gamma} \delta^{k}_l.\end{align*}
Using these canonical co-ordinates we can define linear subspaces $\mathcal{H}_+$ and $\mathcal{H}_-$ to be the spans, respectively, of the $A^{k}_\alpha$ and $B^{\gamma}_l$ inside $\Hcal$:
\begin{align*} \mathcal{H}_{+} & := \HH^*(X)[z] = \left\{ \sum_{k \geq 0} q_k^\alpha \varphi_\alpha z^k \colon q_k^\alpha \in \Lambda \right\} \\
\mathcal{H}_- & := z^{-1} \HH^*(X) \llbracket z^{-1} \rrbracket = \left\{ \sum_{l \geq 0} p^l_\gamma \varphi^\gamma (-z)^{-1-l} \colon p^l_\gamma \in \Lambda \right\}.\end{align*}
Here, and in what follows, we adopt the Einstein summation convention when dealing with Greek letters, i.e. when summing over cohomology classes $\varphi_\alpha$ and $\varphi^\gamma$. It is clear that both $\Hcal_+$ and $\Hcal_-$ are Lagrangian subspaces, in the sense that:
\begin{equation*} \mathcal{H}_\pm^\perp = \left\{ v \in \Hcal \ \bigg| \ \Omega(v,w)=0 \text{ for all } w \in \mathcal{H}_\pm \right\} = \mathcal{H}_\pm .\end{equation*}
Thus we think of $\mathcal{H}_+$ and $\mathcal{H}_-$ as being ``half-dimensional'' or ``semi-infinite'' (since in the finite-dimensional setting a Lagrangian subspace is always half-dimensional). Furthermore this splitting gives an identification of symplectic vector spaces
\begin{equation*} \Hcal = \T^*\mathcal{H}_+ \end{equation*}
which means that $\mathcal{H}_-$ gets identified with the cotangent fibre; in terms of the co-ordinates $q^\alpha_k$, $p_\gamma^l$ above, the identification is:
\begin{equation*} p_\alpha^k = \dfrac{\partial}{\partial q^\alpha_k}. \end{equation*}

\subsection{Lagrangian cone} \label{subsection Lagrangian cone}
We are now in a position to construct the Lagrangian cone $\Lcal_X$. A standard object in Gromov--Witten theory is the \emph{genus-zero descendant potential}, which is a formal generating function for the genus-zero Gromov--Witten invariants:
\begin{equation*} \Fcal^0_X ( \tCoh{z} ) = \sum_{\beta, n} \dfrac{Q^\beta}{n!} \langle \tCoh{\psi_1}, \ldots, \tCoh{\psi_n} \rangle^X_{0,n,\beta}.\end{equation*}
Let us explain the notation above. The sum is over all curve classes $\beta \in \HH_2^+(X)$ and non-negative integers $n \geq 0$. The variable $Q$ is a formal variable, called the \emph{Novikov variable}, which keeps track of the curve class. We make sense of this by taking the ground field $\Lambda$ to be the \emph{Novikov field}:
\begin{equation*} \Lambda = \C (( \HH_2^+(X) )). \end{equation*}
Remember that we defined $\HH^*(X) = \HH^*(X;\Lambda)$ for some unspecified field $\Lambda$; from now on we take $\Lambda$ to be the Novikov field. The parameter $\tCoh{z}$ of the generating function is a formal power series with coefficients in $\HH^*(X)$
\begin{align*} \tCoh{z} & = \sum_{k \geq 0} t_k z^k \qquad \ \ \ \ \, t_k \in \HH^*(X) \\
& = \sum_{k \geq 0} t_k^\alpha \varphi_\alpha z^k \qquad t_k^\alpha \in \Lambda\end{align*}
so that the correlators above are interpreted as
\begin{align*} \langle \tCoh{\psi_1}, \ldots, \tCoh{\psi_n} \rangle^X_{0,n,\beta} := &\, \left\langle \sum_{k_1 \geq 0} t_{k_1}^{\alpha_1} \varphi_{\alpha_1} \psi_1^{k_1} , \ldots, \sum_{k_n \geq 0} t_{k_n}^{\alpha_n} \varphi_{\alpha_n} \psi_n^{k_n} \right\rangle^X_{0,n,\beta}\\
= & \sum_{k_1,\ldots,k_n \geq 0} t_{k_1}^{\alpha_1} \cdots t_{k_n}^{\alpha_n} \langle \varphi_{\alpha_1} \psi_1^{k_1},\ldots,\varphi_{\alpha_n} \psi_n^{k_n} \rangle^X_{0,n,\beta} \end{align*}
(remember that we are using the Einstein summation convention for the Greek letters). Thus we may rewrite $\Fcal^0_X$ in a more transparent (though less convenient) form as:
\begin{equation*} \Fcal^0_X( \tCoh{z} ) = \sum_{\beta,n} \dfrac{Q^\beta}{n!} \sum_{k_1,\ldots,k_n \geq 0} t_{k_1}^{\alpha_1} \cdots t_{k_n}^{\alpha_n} \cdot \langle \varphi_{\alpha_1} \psi_1^{k_1},\ldots,\varphi_{\alpha_n} \psi_n^{k_n} \rangle^X_{0,n,\beta}. \end{equation*}
We view this as a formal power series in the variables $t_k^\alpha$ for $k \geq 0$ and $\alpha = 0,\ldots,N$. Notice that these co-ordinates are indexed by the same set as the co-ordinates $q_k^\alpha$ for $\Hcal_+$ defined in \S \ref{subsection Givental space}; the two are related by the following change of variables
\begin{equation*} \qCoh{z} = \tCoh{z} - z \cohid  \end{equation*}
called the \emph{dilaton shift}. In concrete terms this means that $q_k^\alpha = t_k^\alpha$ unless $(k,\alpha) = (1,0)$, in which case $q_1^0 = t_1^0 - 1$. Under this change of variables, we can view $\Fcal^0_X$ as a function
\begin{equation*} \Fcal^0_X \colon \Hcal_+ \to \Lambda \end{equation*}
and hence the derivative $\diff\!\Fcal^0_X$ defines a section of the cotangent bundle $\T^* \Hcal_+$. The Lagrangian cone is defined as the graph of this section:
\begin{equation*} \Lcal_X := \left\{ (\qCoh{z} , \pCoh{z}) \in \Hcal = \Hcal_+ \oplus \Hcal_-  \ \bigg| \ \pCoh{z} = \diff\!\Fcal^0_X ( \qCoh{z} ) \right\}.\end{equation*}
Thus for every point $\qCoh{z} \in \Hcal_+$ there is a unique point of $\Lcal_X$ lying over $\qCoh{z}$. In concrete terms, this is:
\begin{align*} \Lcal_X|_{\qCoh{z}} & = \left( \tCoh{z} - z\cohid \right) + \sum_{\beta,n} \dfrac{Q^\beta}{n!} \sum_{l \geq 0} \left\langle \tCoh{\psi_1}, \ldots, \tCoh{\psi_n}, \varphi_\gamma \psi_{n+1}^l \right\rangle^X_{0,n+1,\beta} \cdot \varphi^\gamma (-z)^{-l-1} \\
& = \left( \tCoh{z} - z\cohid \right) + \sum_{\beta,n} \dfrac{Q^\beta}{n!} \left\langle \tCoh{\psi_1}, \ldots, \tCoh{\psi_n}, \left(\dfrac{\varphi_\gamma}{-z-\psi_{n+1}} \right) \right\rangle^X_{0,n+1,\beta} \cdot \varphi^\gamma . \end{align*}
The first term $\tCoh{z} - z\cohid = \qCoh{z}$ specifies the point in the base, while the remaining terms specify the point in the fibre. The meaning of the fractional insertion in the second line is that it should be expanded as a power series in $z^{-1}$, the result of which is precisely the expression on the first line.

As it has been presented, divorced from its origins in physics, $\Lcal_X$ may come across as a mysterious object. Working with it takes some getting used to, but the eventual payoff is significant, and it is now recognised as a fundamental tool in Gromov--Witten theory. To give just a taste of this, we state a few basic facts about the Lagrangian cone.

\begin{theorem}[{\cite[Proposition 1]{CoatesGivental}}] \label{theorem properties of the cone} The following basic properties hold:
\begin{itemize}
\item $\Lcal_X$ is a cone (it is preserved under scalar multiplication by elements of $\Lambda$);
\item for $f \in \Lcal_X$, we have $(\T_f\Lcal_X) \cap \Lcal_X = z \cdot \T_f \Lcal_X \subseteq \Hcal$;
\item the set of all tangent spaces to $\Lcal_X$ forms a finite-dimensional family; thus $\Lcal_X$ is ruled by a finite-dimensional family of linear subspaces.
\end{itemize} \end{theorem}
\noindent Thus we see that the geometry of $\Lcal_X$ is very tightly constrained. The above theorem is actually equivalent \cite[Theorem 1]{GiventalSymplecticFrobenius} to the following three fundamental results in Gromov--Witten theory: the string equation, the dilaton equation and the topological recursion relations. More generally, the Lagrangian cone can be used to conveniently express statements which would be exceedingly cumbersome to phrase otherwise. For more on this, see \cite{GiventalQuantisation}, \cite{CoatesIritaniFock}.

Finally, we note that the dilaton shift $\qCoh{z} = \tCoh{z} - z\cohid$ is an essential part of the theory; for instance, $\Lcal_X$ is not even a cone in the $\tCoh{z}$ co-ordinates.

\subsection{Fundamental solution matrix} \label{subsection Fundamental solution matrix} There is one more object in Gromov--Witten theory which we must define. The \emph{fundamental solution matrix} is a family of symplectic operators on the Givental space $\Hcal$ (so named because it encodes a fundamental set of solutions to the quantum differential equations \cite{Dubrovin2D}). For our purposes it depends on a parameter $\qCoh{z} \in \Hcal_+$, and is given by:
\begin{equation*} S_{\tCoh{z}}(f) = f + \sum_{\beta,n} \dfrac{Q^\beta}{n!} \left\langle \left( \dfrac{f}{z-\psi_{0}}\right),\tCoh{\psi_1},\ldots,\tCoh{\psi_n}, \varphi_\gamma \right\rangle^X_{0,n+2,\beta}\cdot \varphi^\gamma .\end{equation*}
Here the insertion $f \in \Hcal$ is expanded linearly in the $z$ and $\varphi_\alpha$, and $\tCoh{z}$ is the dilaton-shifted element corresponding to $\qCoh{z}$ (we write $S_{\tCoh{z}}$ instead of $S_{\qCoh{z}}$ to keep our notation compatible with standard usage). As with the Lagrangian cone, the fundamental solution matrix has deep connections to physics, and has been the focus of intense study. We will not attempt to say more than this here; the interested reader should consult \cite{PandharipandeRationalCurves} and \cite[\S 10]{CoxKatz}.

In this article we will view $S$ as a single endomorphism of the trivial $\Hcal$-bundle over $\Hcal_+$
\bcd
\Hcal_+ \times \Hcal \ar[rr,"S"] \ar[rd] & & \Hcal_+ \times \Hcal \ar[ld] \\
& \Hcal_+ &
\ecd
where the endomorphism $\Hcal \to \Hcal$ over $\qCoh{z} \in \Hcal_+$ is given by $S_{\tCoh{z}}$. We can also view the Lagrangian cone as a submanifold of $\Hcal_+ \times \Hcal$ by doubling the base co-ordinate:
\begin{equation*} \Lcal_X = \left\{ (\qCoh{z},\qCoh{z},\pCoh{z}) \  \bigg| \ \pCoh{z} = \operatorname{d}\!\Fcal^0_X ( \qCoh{z} ) \right\} \subseteq \Hcal_+ \times \Hcal . \end{equation*}
Thus, we can define the transform $S(\Lcal_X) \subseteq \Hcal_+ \times \Hcal$ of $\Lcal_X$ by $S$ without having to specify a parameter $\qCoh{z}$. This will be important for the statement of our main result.

\subsection{Relative stable maps} \label{subsection Relative stable maps} The final ingredient which we need to explain is the theory of relative stable maps. Given a smooth projective variety $Z$ and a smooth hypersurface $Y \subseteq Z$, the moduli space of relative stable maps parametrises stable maps in $Z$ with fixed tangency orders to $Y$ at the marked points. If there are $n$ marked points then this tangency information is encoded in a vector $\alpha=(\alpha_1,\ldots,\alpha_n)$ of non-negative integers. The resulting moduli space
\begin{equation*} \M{0}{\alpha}{Z|Y}{\beta} \end{equation*}
should parametrise stable maps to $Z$ such that the $i$th marked point has tangency order $\alpha_i$ to the divisor $Y$ (by convention, $\alpha_i=0$ means that the marked point is not mapped into the divisor at all, while $\alpha_i=1$ means it is mapped into the divisor transversely; as such, the map only truly becomes ``tangent'' to the divisor when $\alpha_i \geq 2$). This data must satisfy the obvious numerical condition $\Sigma_i \alpha_i = Y \cdot \beta$. The question of how to define these spaces rigorously is a non-trivial one; the problem with the na\"ive approach described above is that the deformation theory can become extremely wild when there are components of the source curve mapping into $Y$; this wildness means that the usual construction of the virtual fundamental class no longer works, so these spaces cannot be used to define invariants.

The earliest solution to this problem, due to J. Li and following ideas first developed in symplectic geometry, is to allow the target $Z$ to degenerate into a so-called expanded degeneration $Z[l]$ \cite{Li1, Li2}. The space $Z[l]$ is constructed from $Z$ by gluing on a chain of $l$ copies of the projective completion of the normal bundle to $Y$ in $Z$:
\begin{equation*} P = \PP_Y \left(\operatorname{N}_{Y|Z} \oplus \OO_Y \right). \end{equation*}
The picture is as follows (which illustrates the case $Z[2]$):

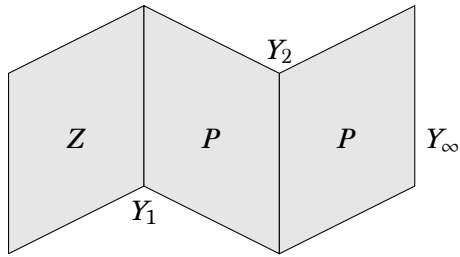
\begin{center}
\begin{tikzpicture}[scale=0.6]
\draw[fill=gray!20] (-5,-2) -- (-5,2) -- (-2,3.5) -- (-2,-0.5) -- (-5,-2);
\draw[fill=gray!20] (-2,3.5) -- (-2,-0.5) -- (1,-2) -- (1,2) -- (-2,3.5);
\draw[fill=gray!20] (1,-2) -- (1,2) -- (4,3.5) -- (4,-0.5) -- (1,-2);

\draw (-3.5,1) node[below]{$Z$};
\draw (-0.5,1) node[below]{$P$};
\draw (2.5,1) node[below]{$P$};

\draw (-2,-0.6) node[below]{$Y_1$};
\draw (1,2) node[above]{$Y_2$};
\draw (4,0.5) node[right]{$Y_\infty$};
\end{tikzpicture}
\end{center}
The idea is that, whenever a component of the source curve starts to fall into the divisor, the target ``bubbles'' off an extra copy of $P$, and the internal component is then mapped (transversely) into $P$. 

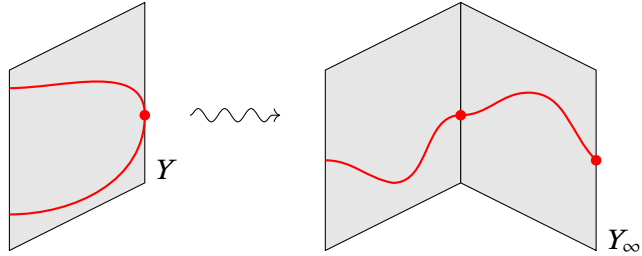
\begin{center}
\begin{tikzpicture}[scale=0.6]
\draw[fill=gray!20] (-12,-2) -- (-12,2) -- (-9,3.5) -- (-9,-0.5) -- (-12,-2);
\draw[thick,color=red] (-12,-1.2) to [out=0,in=270] (-9,1) to [out=90,in=0] (-12,1.6);
\draw[fill=red,color=red] (-9,1) circle[radius=3pt];
\draw (-9,-0.2) node[right]{$Y$};

\draw(1,-1.8) node[right]{$Y_\infty$};

\draw[->,decorate,decoration=snake] (-8,1) -- (-6,1);

\draw[fill=gray!20] (-5,-2) -- (-5,2) -- (-2,3.5) -- (-2,-0.5) -- (-5,-2);
\draw[fill=gray!20] (-2,3.5) -- (-2,-0.5) -- (1,-2) -- (1,2) -- (-2,3.5);
\draw[thick,color=red] (-5,0) to [out=0,in=170] (-3.5,-0.5) to [out=0,in=180] (-2,1);
\draw[fill=red,color=red] (-2,1) circle[radius=3pt];
\draw[thick,color=red] (-2,1) to [out=0,in=180] (-0.5,1.5) to [out=0,in=135] (1,0);
\draw[fill=red,color=red] (1,0) circle[radius=3pt];
\end{tikzpicture}
\end{center}
Two such maps into $P$ are identified if they differ by an element of the group $\C^*$ of automorphisms of $P$ given by rescalings of the fibre. As illustrated above, the resulting map to $Z[l]$ is transverse in a very strong sense: the only points of the curve which map to the infinity divisor are the markings $x_i$, and they do so with the correct tangency order $\alpha_i$. On the other hand, the curve can only map to the singular locus at a finite number of isolated nodal points, and for each node the tangency orders of the two adjacent branches of the curve to the singular locus must be equal. This transversality condition, usually called predeformability, ensures that the resulting moduli space has the correct virtual dimension. An extremely careful analysis of the deformation theory of this new space then shows that a virtual class can be defined \cite{Li2}. Integrals against this virtual class are called \emph{relative Gromov--Witten invariants} of $(Z,Y)$. In our applications we will always have $Z=X \times \PP^1$ and $Y=X\times \infty$. In this case the normal bundle of $Y$ in $Z$ is trivial, so $P \cong X \times \PP^1 = Z$ and thus all the levels of the expanded degeneration, including level $0$, are isomorphic.

We will assume that the reader is reasonably familiar with relative stable maps; all the facts which we will use can be found in \S\S 2-3 of \cite{GraberVakil}, which also serves as a good introduction to relative Gromov--Witten theory.

\begin{remark} \label{remark Log} More recently, the theory of \emph{logarithmic stable maps}, as developed by D. Abramovich, Q. Chen, M. Gross and B. Siebert,  has provided an alternative (and significantly more general) approach to relative stable maps \cite{GrossSiebertLog} \cite{ChenLog} \cite{AbramovichChenLog}. We expect that the computations we carry out here will carry over to the log setting, once a suitable localisation formula has been established for log stable maps. Indeed, log Gromov--Witten theory relative a simple normal crossings divisor seems to be the correct generality in which to apply the construction given in this article. \end{remark}

\subsection{Statement of the main result} \label{subsection Statement of the proposition} We are finally in a position to state our main result. Let $X$ be a smooth projective variety. For $\beta \in \HH_2^+(X)$ and $n \geq 0$, consider the moduli space
\begin{equation*} \overline{\mathcal{M}}_{0,n,(1)}\left( (X\times\PP^1 \ | \ X \times \infty), (\beta,1) \right) \end{equation*}
of relative stable maps to $(X\times\PP^1,X \times \infty)$ of class $(\beta,1)$, where the first $n$ marked points $x_1,\ldots,x_n$ have tangency $0$ with the divisor, and the last marked point $x_\infty$ has tangency $1$. There is a natural $\C^*$-action on this moduli space induced by the action on the target $X\times \PP^1$ (acting trivially on the first factor and with weight $-1$ on the second). Consider the following class in the equivariant cohomology of the moduli space
\begin{equation*} \Theta_{\beta,n}(\tCoh{z}) =  (-z) \cdot \prod_{i=1}^n \ev_i^*(\tCoh{\psi_i}) \end{equation*}
where $z$ is the equivariant parameter. Here each $\ev_i$ is viewed as mapping into $X$, via the composition:
\begin{equation*} \overline{\mathcal{M}}_{0,n,(1)}\left( (X\times\PP^1 \ | \ X \times \infty), (\beta,1) \right) \xrightarrow{\ev_i} X \times \PP^1 \xrightarrow{\pi_1} X.\end{equation*}
(Note that this morphism is equivariant with respect to the trivial action on $X$.) We then have:
\begin{proposition}\label{main result}
\begin{equation} \label{main equation} (\ev_\infty)_* \left( \sum_{\beta,n} \dfrac{Q^\beta}{n!} \cdot \Theta_{\beta,n}(\tCoh{z}) \right) = S(\Lcal_X)|_{\qCoh{z}} \end{equation}
where $\qCoh{z}$ is the dilaton-shifted co-ordinate corresponding to $\tCoh{z}$.
\end{proposition}
\noindent The proof will be given in \S \ref{section Proof of the proposition}; for the moment let us explain the statement. We view $\ev_\infty$ as a map
\begin{equation*} \ev_\infty \colon \coprod_{\beta,n} \overline{\mathcal{M}}_{0,n,(1)}\left( (X\times\PP^1\ |\ X \times \infty), (\beta,1) \right) \longrightarrow X\times\infty=X \end{equation*}
so that the target of the push-forward $(\ev_\infty)_*$ is the equivariant cohomology of $X$ with respect to the trivial torus action. But this is just:
\begin{equation*} \HH^*(X) \otimes \Lambda[z] = \Hcal_+ \subseteq \Hcal \end{equation*}
On the other hand, $S(\Lcal_X)$ naturally lives inside the total space of the trivial bundle $\Hcal_+ \times \Hcal \to \Hcal_+$ (see the discussion at the end of \S \ref{subsection Fundamental solution matrix} above); therefore when we write $S(\Lcal_X)$ in equation \eqref{main equation}, we really mean its projection along $\pi_2 \colon \Hcal_+ \times \Hcal \to \Hcal$. Another way to say this is that for a fixed $\qCoh{z} \in \Hcal_+$, with dilaton-shifted co-ordinate $\tCoh{z}$, the push-forward of the left-hand side of \eqref{main equation} is equal to $S_{\tCoh{z}}(\Lcal_X|_{\qCoh{z}})$.

An immediate corollary of the above result is that $S(\Lcal_X) \subseteq z \cdot \Hcal_+$ rather than just $\Hcal$. For an application of this, as well as a deeper exploration of the ``hidden polynomiality'' arising from our construction, see \S \ref{subsection Variants and applications}.

\begin{remark} The total transform $S(\Lcal_X)$ has a geometric interpretation as a family of \emph{ancestor cones}; see \cite[Appendix 2]{CoatesGivental}.\end{remark}

\begin{remark} Notice that for any choice of $\beta$, the curve class $(\beta,1)$ is non-zero. Hence the sum in Proposition \ref{main result} is over \emph{all} $\beta$ and $n$. This is in contrast to the sum which appears in the definition of the Lagrangian cone in \S \ref{subsection Lagrangian cone}, which is only over the stable range, i.e. excludes the cases $(\beta,n) = (0,0)$ and $(0,1)$. This difference will become important during the proof of Proposition \ref{main result}. \end{remark}

\section{Proof of the main result} \label{section Proof of the proposition} We will assume that the reader is familiar with the space of relative stable maps, and in particular with the torus localisation formula, established in \cite{GraberVakil} whenever the divisor is fixed pointwise by the action (as is the case for us). We will write $X_0$ and $X_\infty$ for $X \times 0$ and $X \times \infty$, viewing them either as divisors in $X \times \PP^1$ or in $X[l]$, as appropriate.

\subsection{Identifying the fixed loci} \label{subsection Identifying the fixed loci} The proof proceeds by $\C^*$-localisation. The $\C^*$-fixed loci of the moduli space are indexed by graphs of the following form:\medskip
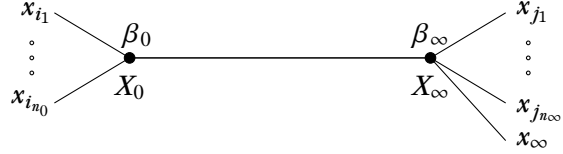
\begin{center}
\begin{tikzpicture}
\draw[fill=black] (-2,0) -- (2,0);
\draw[fill=black] (-2,0) circle[radius=2pt];
\draw (-2,-0.1) node[below]{$X_0$};
\draw (-1.9,0) node[above]{$\beta_0$};
\draw (-3,0.6) -- (-2,0);
\draw (-2.9,0.6) node[left]{$x_{i_1}$};
\draw (-3,-0.6) -- (-2,0);
\draw (-2.9,-0.6) node[left]{$x_{i_{n_0}}$};
\draw (-3.3,-0.2) circle[radius=0.7pt];
\draw (-3.3,0) circle[radius=0.7pt];
\draw (-3.3,0.2) circle[radius=0.7pt];

\draw[fill=black] (2,0) circle[radius=2pt];
\draw (2,-0.1) node[below]{$X_\infty$};
\draw (2,0) node[above]{$\beta_\infty$};
\draw (3,0.6) -- (2,0);
\draw (3,0.6) node[right]{$x_{j_1}$};
\draw (3,-0.6) -- (2,0);
\draw (3,-0.7) node[right]{$x_{j_{n_\infty}}$};
\draw (3.3,-0.2) circle[radius=0.7pt];
\draw (3.3,0) circle[radius=0.7pt];
\draw (3.3,0.2) circle[radius=0.7pt];
\draw (3,-1.1) -- (2,0);
\draw (3,-1.1) node[right]{$x_\infty$};
\end{tikzpicture}
\end{center}
These correspond to splittings of the source curve into three pieces: a piece $C_0$ which maps to $X_0$, a piece $C_\infty$ which maps to $X_\infty$ (and hence, in general, into the higher levels of the expanded degeneration), and a rational component joining $C_0$ and $C_\infty$, which maps isomorphically onto a $\PP^1$-fibre of $X \times \PP^1$. The marking $x_\infty$ always belongs to $C_\infty$ since it must map to the infinity divisor $X_\infty$. The other choices -- of degrees $\beta_0$ and $\beta_\infty$ for the two pieces, and of a partition $A_0 \sqcup A_\infty = \{ x_1,\ldots,x_n\}$ of the non-relative markings -- are free. The fixed locus corresponding to this data is isomorphic to
\begin{equation}\label{Fixed locus} \M{0}{A_0 \cup \{q_0\}}{X}{\beta_0} \times_X \overline{\mathcal{M}}_{0,A_\infty,(1),(1)}\left( X\times \PP^1\ |\ (X_0 + X_\infty),\beta_\infty\right)_{\sim} \end{equation}
with virtual fundamental class induced by the virtual classes of the two factors; this is part of the statement of the virtual localisation theorem in \cite{GraberVakil}. Here the second factor
\begin{equation*}\overline{\mathcal{M}}_{0,A_\infty,(1),(1)}\left( X\times \PP^1\ |\ (X_0 + X_\infty),\beta_\infty\right)_{\sim}\end{equation*}
is a moduli space of stable maps to the non-rigid target; see \cite[\S 2.4]{GraberVakil} for a detailed discussion of this space. The notation here is supposed to indicate that there is a set $A_\infty$ of non-relative markings (so $\# A_\infty=n_\infty$), a single marking $q_\infty$ mapping to $X_0$ with tangency $1$, and a single marking $x_\infty$ mapping to $X_\infty$ with tangency $1$.

The fibre product in \eqref{Fixed locus} is taken with respect to the evaluations at $q_0$ and $q_\infty$ on each side. The Euler class of the virtual normal bundle is equal \cite[Theorem~3.6 and Example~3.7]{GraberVakil} to
\begin{equation*} (-z)(-z-\psi_{q_0})(z-\psi_{q_\infty}) \end{equation*}
which obviously splits into a product of classes supported on the two factors. We should briefly explain these: $-z$ arises from the deformations of the map on the rational bridge, $-z-\psi_{q_0}$ arises from the smoothing of the node connecting the rational bridge to $C_0$ and $z-\psi_{q_\infty}$ is a target psi class, which arises from the smoothing of the target singularity connecting the level $0$ piece and the level $1$ piece of the expanded degeneration. Here we have used the identification of the target psi class with a multiple of the psi class on one of the relative markings \cite[Construction 5.1.17]{GathmannThesis}. The term arising from the smoothing of the node connecting the rational bridge to $C_\infty$ is cancelled out by the local obstruction at that node: see \cite[\S 3.8]{GraberVakil}. 

Note that for certain choices of $(\beta_0, A_0 \ | \ \beta_\infty, A_\infty)$ the moduli spaces which we have written down above do not exist, because the data defining them is not stable. In these degenerate cases, we still have fixed loci; it is simply that one (or both) of the factors becomes trivial. Hence we must deal with these separately. The possible situations are enumerated below.

\subsubsection*{Case 1: $(\beta,n)=(0,0)$} This is the maximally degenerate case. The fixed locus is just $X$, which has virtual codimension $0$; there is no virtual normal bundle.

\subsubsection*{Case 2: $(\beta,n)=(0,1)$ and $n_\infty=0$} In this case the fixed locus is again just $X$, with a single marked point $x_1$ mapped to $X_0$ and another marked point $x_\infty$ mapped to $X_\infty$ (there is no expansion of the target). The virtual codimension is $1$, and the Euler class of the virtual normal bundle is $-z$.

\subsubsection*{Case 3: $n \geq 1$ and $(\beta_0,n_0)=(0,0)$} In this case the fixed locus is a moduli space of relative maps to the non-rigid target, with $n+2$ marked points. The virtual codimension is $1$, and the virtual normal bundle contribution is $z-\psi_{q_\infty}$.

\subsubsection*{Case 4: $n \geq 1$ and $(\beta_0,n_0)=(0,1)$} Here the fixed locus is the same as the one in the previous case, but it now has virtual codimension $2$ because there is a marked point at the $X_0$ end of the rational bridge; the Euler class of the virtual normal bundle is $-z(z-\psi_{q_\infty})$.

\subsubsection*{Case 5: $n \geq 2$ and $(\beta_\infty,n_\infty)=(0,0)$} In this case the fixed locus is just the moduli space of stable maps to $X$ with $n+1$ markings. The virtual codimension is $2$, and the Euler class of the virtual normal bundle is $-z(-z-\psi_{q_0})$.

\subsection{Comparison lemma for psi classes} \label{subsection Lemma on psi classes} We now need to calculate the contributions to the push-forward from each of these fixed loci. A priori this is difficult, because the fixed loci involve moduli spaces of relative stable maps to the non-rigid target, which are in general hard to understand. However, in genus zero, a result of A. Gathmann says that these moduli spaces are in fact virtually birational to the underlying moduli spaces of stable maps to $X$. To be more precise: there is a projection map
\begin{equation*} \pi \colon \overline{\mathcal{M}}_{0,n_\infty,(1),(1)}\left( X\times \PP^1|(X_0 + X_\infty),\beta_\infty\right)_{\sim} \to \M{0}{n_\infty+2}{X}{\beta_\infty} \end{equation*}
induced by the collapsing map from the non-rigid target to $X$, and \cite[Theorem 5.2.7]{GathmannThesis} shows that this map respects the virtual classes:
\begin{equation*} \pi_* \virt{\overline{\mathcal{M}}_{0,n_\infty,(1),(1)}\left( X\times \PP^1|(X_0 + X_\infty),\beta_\infty\right)_{\sim}} = \virt{\M{0}{n_\infty+2}{X}{\beta_\infty}}. \end{equation*}
This result goes a long way towards making these invariants computable. However there is still a problem: the map $\pi$ may contract many components of the source curve, and hence does not in general preserve the psi classes. Consequently, descendant invariants (which certainly appear in our discussion) are still complicated to compute, because one has to keep track of how psi classes pull back. It turns out, however, that $X \times \PP^1$ is special in this respect.

\begin{lemma} \label{psi class lemma} The map $\pi$ cannot contract any component of the source curve which contains a marking.\end{lemma}
\begin{proof} The components contracted by $\pi$ are those with two or fewer special points which are mapped into a fibre of $P=X \times \PP^1$ over $X$. Let $C^\prime$ be such a component. Since it has two or fewer special points, the map $f$ must be non-constant on $C^\prime$ (by stability), and hence there is at least one point of $C^\prime$ which maps to $X_\infty$ and at least one point which maps to $X_0$. Thus, $C^\prime$ contains exactly two special points, which must map to the special divisors of the non-rigid target.

Now suppose for a contradiction that some marking $x_i$ belongs to $C^\prime$. If $x_i$ is a non-relative marking then we immediately arrive at a contradiction, since such a marking cannot map into any special divisor. Otherwise, $x_i = q_\infty$ or $x_\infty$ and so is mapped into $X_0$ or $X_\infty$, respectively; without loss of generality we may suppose $x_i = q_\infty$. By the stability condition for relative stable maps, there must exist some other component of the source curve which maps with positive degree into the same level of the non-rigid target as $C^\prime$. But this would necessarily touch $X_0$, which is a contradiction since $q_\infty$ is the only point of the source curve which is allowed to map to $X_0$ (here we are using the fact that $X \times \PP^1$ is a global product; for non-trivial $\PP^1$-bundles over $X$, it is no longer true that a component of the source curve which touches $X_\infty$ must also touch $X_0$).\end{proof}

\begin{corollary} $\pi^* \psi_i = \psi_i$ for any $i \in \{1,\ldots,n_\infty+2\}$. Thus, we can identify any non-rigid invariant of $(X \times \PP^1,X_0 + X_\infty)$ with the corresponding invariant of $X$. \end{corollary}

\subsection{Calculating the contributions} We are now in a position to calculate the contributions to the push-forward. We fix $(\beta,n)$ and look at the fixed loci of the corresponding moduli space. Ignoring the degenerate cases for the moment, we must sum over \emph{stable splittings} $(\beta_0,A_0 \ | \ \beta_\infty,A_\infty)$ of $(\beta,n)$. We may phrase this as summing over splittings $(\beta_0,\beta_\infty)$ of $\beta$ and $(n_0,n_\infty)$ of $n$, with a factor of ${n\choose n_0} = {n \choose n_\infty}$ introduced to account for the choice of which marked points to put in $A_0$ and which to put in $A_\infty$. Thus the contribution
\begin{align*} \dfrac{Q^\beta}{n!} &  (\ev_\infty)_* \left( (-z) \cdot\prod_{i=1}^n \ev_i^*(\tCoh{\psi_i}) \right) \end{align*}
from the non-degenerate loci is equal to:
\begin{align*}&\ \ \ \ \dfrac{Q^\beta}{n!} \sum_{\substack{\beta_0+\beta_\infty=\beta \\ n_0+n_\infty=n}} {n \choose n_\infty} \langle \tCoh{\psi_1},\ldots,\tCoh{\psi_{n_0}}, \left(\dfrac{\varphi_\alpha}{-z-\psi_{q_0}}\right) \rangle^X_{0,n_0+1,\beta_0} \cdot \\
& \qquad\qquad\qquad \ \ \ \ \ \ \ \ \langle \left(\dfrac{\varphi^\alpha}{z-\psi_{q_\infty}}\right), \tCoh{\psi_1},\ldots,\tCoh{\psi_{n_\infty}}, \varphi_\gamma \rangle^X_{0,n_\infty+2,\beta_\infty} \cdot \varphi^\gamma \\
= & \sum_{\substack{\beta_0+\beta_\infty=\beta \\ n_0+n_\infty=n}} \left( \dfrac{Q^{\beta_0}}{n_0!} \langle \tCoh{\psi_1},\ldots,\tCoh{\psi_{n_0}}, \left(\dfrac{\varphi_\alpha}{-z-\psi_{q_0}}\right) \rangle^X_{0,n_0+1,\beta_0}\right) \cdot \\
& \qquad\qquad\qquad\ \ \ \ \ \ \ \left( \dfrac{Q^{\beta_\infty}}{n_\infty!}\langle \left(\dfrac{\varphi^\alpha}{z-\psi_{q_\infty}}\right), \tCoh{\psi_1},\ldots,\tCoh{\psi_{n_\infty}}, \varphi_\gamma \rangle^X_{0,n_\infty+2,\beta_\infty} \cdot \varphi^\gamma \right).
\end{align*}
There are also the contributions from the degenerate fixed loci, enumerated in \S \ref{subsection Identifying the fixed loci} above. We  now calculate these.

\subsubsection*{Case 1: $(\beta,n)=(0,0)$} This gives a single contribution, which is:
\begin{equation*} -z (\ev_\infty)_* ( \cohid) = -z \cohid. \end{equation*}

\subsubsection*{Case 2: $(\beta,n)=(0,1)$ and $n_\infty=0$} This also gives a single contribution, which is
\begin{equation*} (\ev_\infty)_* ( \ev_1^* \tCoh{\psi_1} ) = \tCoh{z} \end{equation*}
here we have used the fact that the psi class $\psi_1$ restricts to a trivial class on the fixed locus with non-trivial weight $z$, so the equivariant class $\psi_1$ gets identified with $z$.

\subsubsection*{Case 3: $n \geq 1$ and $(\beta_0,n_0)=(0,0)$} Here we get a contribution for each $(\beta,n)$ with $n \geq 1$. The contribution is:
\begin{equation*} \dfrac{Q^{\beta_\infty}}{n_\infty!} \langle \left(\dfrac{-z\cohid}{z-\psi_{q_\infty}} \right), \tCoh{\psi_1},\ldots,\tCoh{\psi_{n_\infty}}, \varphi_\gamma \rangle^X_{0,n_\infty+2,\beta} \cdot \varphi^\gamma . \end{equation*}

\subsubsection*{Case 4: $n \geq 1$ and $(\beta_0,n_0)=(0,1)$} We get a contribution for each $(\beta,n)$ with $n \geq 1$, and the contribution is
\begin{equation*} \dfrac{Q^{\beta_\infty}}{n_\infty!} \langle \left(\dfrac{\tCoh{z}}{z-\psi_{q_\infty}} \right), \tCoh{\psi_1},\ldots,\tCoh{\psi_{n_\infty}}, \varphi_\gamma \rangle^X_{0,n_\infty+2,\beta_\infty} \cdot \varphi^\gamma \end{equation*}
where again we have used the fact that the class $\psi_0$ restricts to the pure weight class $z$ on the fixed locus.

\subsubsection*{Case 5: $n \geq 2$ and $(\beta_\infty,n_\infty)=(0,0)$} Here we get a contribution for each $(\beta,n)$ with $n \geq 2$, and the contribution is:
\begin{equation*} \dfrac{Q^{\beta_0}}{n_0!} \langle \tCoh{\psi_1},\ldots,\tCoh{\psi_{n_0}}, \left( \dfrac{\varphi_\gamma}{-z-\psi_{q_0}} \right) \rangle^X_{0,n_0+1,\beta_0} \cdot \varphi^\gamma . \end{equation*}

\subsection{Putting everything together}
If we sum together all the terms computed in the previous section, we obtain:
\begin{align*} \left( \tCoh{z} - z\cohid \right) + & \sum_{\beta_0,n_0} \dfrac{Q^{\beta_0}}{n_0!} \langle \tCoh{\psi_1},\ldots,\tCoh{\psi_{n_0}}, \left( \dfrac{\varphi_\alpha}{-z-\psi_{q_0}} \right) \rangle^X_{0,n_0+1,\beta_0} \cdot \varphi^\alpha \\
+ & \left( \sum_{\beta_0,n_0} \dfrac{Q^{\beta_0}}{n_0!} \langle \tCoh{\psi_1},\ldots,\tCoh{\psi_{n_0}}, \left( \dfrac{\varphi_\alpha}{-z-\psi_{q_0}} \right) \rangle^X_{0,n_0+1,\beta_0} \right) \cdot \\
& \left( \sum_{\beta_\infty,n_\infty} \dfrac{Q^{\beta_\infty}}{n_\infty!} \langle \left( \dfrac{\varphi^\alpha}{z-\psi_{q_\infty}} \right) , \tCoh{\psi_1},\ldots,\tCoh{\psi_{n_\infty}}, \varphi_\gamma \rangle^X_{0,n_\infty+2,\beta_\infty}  \cdot \varphi^\gamma \right) \\
+ & \sum_{\beta_\infty,n_\infty} \dfrac{Q^{\beta_\infty}}{n_\infty!} \langle \left(\dfrac{\tCoh{z}-z\cohid}{z-\psi_{q_\infty}} \right), \tCoh{\psi_1},\ldots,\tCoh{\psi_n}, \varphi_\gamma \rangle^X_{0,n_\infty+2,\beta_\infty} \cdot \varphi^\gamma . \end{align*}
Using $\qCoh{z} = \tCoh{z} - z\cohid$ and grouping the final two terms together, we see that this is equal to:
\begin{align*} \qCoh{z} + & \sum_{\beta_0,n_0} \dfrac{Q^{\beta_0}}{n_0!} \langle \tCoh{\psi_1},\ldots,\tCoh{\psi_{n_0}}, \left( \dfrac{\varphi_\alpha}{-z-\psi_{q_0}} \right) \rangle^X_{0,n_0+1,\beta_0} \cdot \varphi^\alpha \\
& + \sum_{\beta_\infty,n_\infty} \dfrac{Q^{\beta_\infty}}{n_\infty!} \Bigg\langle \dfrac{1}{z-\psi_{q_\infty}}\cdot \left( \qCoh{z} + \sum_{\beta_0,n_0} \dfrac{Q^{\beta_0}}{n_0!} \langle \tCoh{\psi_1},\ldots,\tCoh{\psi_{n_0}}, \left( \dfrac{\varphi_\alpha}{-z-\psi_{q_0}} \right) \rangle^X_{0,n_0+1,\beta_0} \cdot \varphi^\alpha \right), \\
& \qquad \qquad  \qquad \qquad \qquad \qquad \qquad \tCoh{\psi_1},\ldots,\tCoh{\psi_{n_\infty}}, \varphi_\gamma \Bigg\rangle^X_{0,n_\infty+2,\beta_\infty} \cdot \varphi^\gamma . \end{align*}
But this is equal to
\begin{align*} & \Lcal_X|_{\qCoh{z}} + \sum_{\beta_\infty,n_\infty} \dfrac{Q^{\beta_\infty}}{n_\infty!} \langle \left( \dfrac{\Lcal_X|_{\qCoh{z}}}{z-\psi_{q_\infty}} \right), \tCoh{\psi_1},\ldots,\tCoh{\psi_{n_\infty}}, \varphi_\gamma \rangle^X_{0,n_\infty+2,\beta_\infty} \cdot \varphi^\gamma = S(\Lcal_X)|_{\qCoh{z}}.  \end{align*}
as claimed. This completes the proof of Proposition \ref{main result}. 

\begin{remark} It is perhaps worth comparing our computation to the computation carried out in \cite{CoatesLagrangianConeS1}. There, the moduli space under consideration is the space of ordinary stable maps to $X \times \PP^1$;  Coates restricts to an open substack of this space, consisting of stable maps such that only a single point of the curve is mapped to $X_\infty$. He then applies torus localisation and pushes forward from the (proper) fixed loci. From our point of view, the loci from which he pushes forward are the degenerate loci which appear as Case 5 in \S \ref{subsection Identifying the fixed loci} above. The special cases which he calls Case 2 and Case 3 are what we call Case 2 and Case 1, respectively. Our non-special case, which contributes a product of invariants from stable maps to $X$ and stable maps to the non-rigid target, does not appear in his setting; nor do our special cases 3 and 4. \end{remark}

\section{Variants and applications} \label{subsection Variants and applications} Since an equivariant push-forward must take values in $\HH^*(X) \otimes \Lambda[z] = \Hcal_+$, an immediate consequence of Proposition \ref{main result} is the following:
\begin{theorem} \label{polynomiality for S and L} $S(\Lcal_X) \subseteq z \cdot \Hcal_+$.\end{theorem}
\noindent This is somewhat surprising, since a priori we only know that $S(\Lcal_X) \subseteq \Hcal$, and indeed both $S$ and $\Lcal_X$ involve many non-positive powers of $z$. What Theorem \ref{polynomiality for S and L} says is that the coefficients of these non-positive powers cancel out when we take $S(\Lcal_X)$; this translates into a sequence of universal relations for the Gromov--Witten invariants. Calculating the coefficients of $z^{-k}$ explicitly, we obtain for $k \geq 2$ and $\qCoh{z} \in \Hcal_+$
\begin{align*} \bigg( \langle\!\langle \psi_1^{k-1} \qCoh{\psi_1}, & \varphi_\alpha  \rangle\!\rangle^X_{0,2} \ + \  (-1)^k \langle\!\langle \varphi_\alpha\psi_1^{k-1} \rangle\!\rangle^X_{0,1} + \\
& \sum_{r=0}^{k-2} (-1)^{1+r} \langle\!\langle \varphi_\gamma \psi_1^r \rangle\!\rangle^X_{0,1} \cdot \langle\!\langle \varphi^\gamma \psi_1^{k-2-r},\varphi_\alpha \rangle\!\rangle^X_{0,2} \bigg) (\tCoh{\psi}) \cdot \varphi^\alpha =0\end{align*}
where we have used the correlator notation:
\begin{equation*} \langle\!\langle \varphi_{\alpha_1}\psi_1^{k_1},\ldots,\varphi_{\alpha_r}\psi_r^{k_r} \rangle\!\rangle^X_{0,r}(\tCoh{\psi}) := \sum_{\beta,n} \dfrac{Q^\beta}{n!} \langle \varphi_{\alpha_1}\psi_1^{k_1},\ldots,\varphi_{\alpha_r}\psi_r^{k_r} , \tCoh{\psi_{r+1}},\ldots,\tCoh{\psi_{r+n}} \rangle^X_{0,n+r,\beta}. \end{equation*}
These equations appear to be equivalent to the reconstruction relation \cite[Equation (2)]{LeePandharipande}, combined with the dilaton equation.

\begin{remark} Theorem \ref{polynomiality for S and L} can be viewed as a generalisation of one of the fundamental results in the quantisation formalism, namely that the $J$-function is inverse to the fundamental solution matrix; see Remark \ref{comparison with t=tau case} below.
\end{remark}

In this section we will now extend the above line of argument, exploiting the ``hidden polynomiality'' implicit in our construction. We obtain new proofs and generalisations of several foundational results concerning both the fundamental solution matrix and the Lagrangian cone.

\subsection{The fundamental solution matrix and its adjoint} Looking at the definition given in \S \ref{subsection Fundamental solution matrix}, we see that we can regard $S_{\tCoh{z}}$ as a power series in $z^{-1}$ with coefficients in $\operatorname{End}(\HH^*(X))$:
\begin{equation*} S_{\tCoh{z}} \in \operatorname{End}(\HH^*(X))\llbracket z^{-1} \rrbracket .\end{equation*}
We will write $S_{\tCoh{z}}(z)$ to emphasise this point of view. The adjoint ${S_{\tCoh{z}}}^*(z)$ is defined by taking the adjoints, term-by-term, of the coefficients of $S_{\tCoh{z}}(z)$ (with respect to the Poincar\'e pairing on $\HH^*(X)$). It is easy to check that, for $v \in \HH^*(X)$:
\begin{equation} \label{S* equation} {S_{\tCoh{z}}}^*(z)(v) = v + \sum_{\beta,n} \dfrac{Q^\beta}{n!} \langle v , \tCoh{\psi_1}, \ldots, \tCoh{\psi_n} , \left( \dfrac{\varphi_\alpha}{z-\psi} \right) \rangle^X_{0,n+2,\beta} \cdot \varphi^\alpha . \end{equation}
An important feature of the theory \cite{GiventalEquivariant} is that when $\tCoh{z}=\tau$, the operators $S_\tau(z)$ and ${S_\tau}^*(-z)$ are inverse to each other; this is in fact equivalent to the statement that $S_\tau(z)$ is a symplectomorphism \cite[\S 3.1]{CladerPriddisShoemaker}. We now generalise this fact to arbitrary $\tCoh{z}$, based on a slight modification of the construction used in Proposition \ref{main result}.
\begin{proposition} \label{fsm transpose inverse} ${S_{\tCoh{z}}}^*(-z) = S_{\tCoh{z}}(z)^{-1}$.\end{proposition}
\begin{proof} We first note that it is sufficient to prove:
\begin{equation} \label{equation one composition} S_{\tCoh{z}}(z)\circ {S_{\tCoh{z}}}^*(-z) = \operatorname{Id}_{\HH^*(X)} . \end{equation}
Indeed, the operators $S_{\tCoh{z}}(z)$ and ${S_{\tCoh{z}}}^*(-z)$ can  be viewed as finite-dimensional matrices over the field of Laurent series $\Lambda((z^{-1}))$. If \eqref{equation one composition} holds then both these matrices have maximal rank, and therefore we also have:
\begin{equation*} {S_{\tCoh{z}}}^*(-z) \circ S_{\tCoh{z}}(z) = \operatorname{Id}_{\HH^*(X)}.\end{equation*}
Thus it remains to show \eqref{equation one composition}. We consider the following moduli space
\begin{equation*} \overline{\mathcal{M}}_{0,n,(1),(1)}\left( (X\times\PP^1 \ | \ X_0 + X_\infty), (\beta,1) \right) \end{equation*}
which has a single marked point $x_0$ mapping to $X_0$, a single marked point $x_\infty$ mapping to $X_\infty$, and a collection of other markings $x_1,\ldots,x_n$ which carry no tangency conditions.

Since the divisor is now disconnected, we must be slightly careful about what we mean by the space above. For our purposes, the allowed automorphisms act separately on the fibres of the expanded degeneration over $X_0$ and $X_\infty$. The stability condition is also imposed separately. As such, each expansion is now indexed by two integers, $l_0$ and $l_\infty$, giving the lengths of the expansion over $X_0$ and $X_\infty$ respectively. This is close to the approach taken in \cite{FaberPandharipandeRelative}. One can view this moduli space as the fibre product:
\begin{equation*} \overline{\mathcal{M}}_{0,n+1,(1)}\left( X\times\PP^1|X_0 , (\beta,1) \right) \times_{\M{0}{n+2}{X\times\PP^1}{(\beta,1)}} \overline{\mathcal{M}}_{0,n+1,(1)}\left( X\times\PP^1|X_\infty , (\beta,1) \right).\end{equation*}
Taking the definition this way ensures that, when we localise, the fixed loci are fibre products of moduli spaces of relative stable maps to the non-rigid target. Furthermore since the stability condition is imposed separately over $X_0$ and $X_\infty$, the proof of Lemma \ref{psi class lemma} still applies. An analogous computation to the one given in \S \ref{section Proof of the proposition} then shows that, for $v \in \HH^*(X)$:
\begin{equation*} (\ev_\infty)_* \left( \sum_{\beta,n} \dfrac{Q^\beta}{n!} \cdot \ev_0^*(v) \cdot \prod_{i=1}^n \ev_i^* \tCoh{\psi_i} \right) = S_{\tCoh{z}}(z) \left( {S_{\tCoh{z}}}^*(-z) (v) \right). \end{equation*}
Since this is an equivariant push-forward, we see that $S_{\tCoh{z}}(z) \circ {S_{\tCoh{z}}}^*(-z)$ is a polynomial in $z$ with coefficients in $\operatorname{End}(\HH^*(X))$. On the other hand it is obvious from the definitions that it is also a power series in $z^{-1}$. Thus $S_{\tCoh{z}}(z) \circ {S_{\tCoh{z}}}^*(-z)$ is constant in $z$, and since the constant term is clearly the identity this completes the proof. \end{proof}

\begin{remark} \label{comparison with t=tau case} As noted previously, Proposition \ref{fsm transpose inverse} is a generalisation of the following fundamental fact for $\tau \in \HH^*(X)$:
\begin{equation*} {S_\tau}^*(-z) = S_\tau(z)^{-1}.\end{equation*}
I would like to thank M. Shoemaker for pointing out that one can also view Theorem \ref{polynomiality for S and L} as a generalisation of this result. Indeed, when $\tCoh{z}=\tau$ we can use the string equation to show that
\begin{equation}\label{L is S*} \Lcal_X|_{\qCoh{z}} = {S_{\tau}}^*(-z)(-z) \end{equation}
where $\qCoh{z}=\tau-z$. Thus we find:
\begin{equation*} S(\Lcal_X)|_{\qCoh{z}} = S_\tau(\Lcal_X|_{\qCoh{z}}) = S_\tau(z) \circ {S_\tau}^*(-z)(-z) = -z \in z \cdot \Hcal_+ . \end{equation*}
Our result can be viewed as a generalisation of this to arbitrary $\tCoh{z}$. The original proof does not apply in this more general setting, because it relies on an application of the string equation which produces additional unwanted terms when $\tCoh{z}$ involves higher powers of $z$. In particular, the identification \eqref{L is S*} no longer holds, which explains why we end up with two different generalisations.\end{remark}

\subsection{Properties of the Lagrangian cone} Here we reprove two fundamental facts concerning the Lagrangian cone. First, we modify the previous construction to give a concrete proof that $\Lcal_X$ is Lagrangian (though it should be noted that this also follows from the general fact that the graph of any closed $1$-form is Lagrangian).

\begin{proposition} \label{prop L a cone} $\Lcal_X$ is Lagrangian. \end{proposition}
\begin{proof} Let $\qCoh{z} \in \Hcal_+$ be a point in the base and let $f = \Lcal_X|_{\qCoh{z}} \in \Hcal$ be the  point on the cone lying over $\qCoh{z}$. We must show that $\T_f \Lcal_X$ is a Lagrangian subspace of $\Hcal$. First let us describe the points of $\T_f \Lcal_X$. Recall that $f$ is given by:
\begin{equation*} f= \Lcal_X|_{\qCoh{z}} = \qCoh{z} + \sum_{\beta,n} \dfrac{Q^\beta}{n!} \langle \tCoh{\psi_1},\ldots,\tCoh{\psi_n}, \left( \dfrac{\varphi_\gamma}{-z-\psi} \right) \rangle^X_{0,n+1,\beta} \cdot \varphi^\gamma . \end{equation*}
Since $\Lcal_X$ is the graph of the section $\diff\!\Fcal^0_X$, the tangent space $\T_f\Lcal_X$ is spanned by the partial derivatives of the above expression in the $\Hcal_+$-co-ordinates. Given such a co-ordinate $q_k^\alpha$ the corresponding derivative is:
\begin{equation*} \varphi_\alpha z^k + \sum_{\beta,n} \dfrac{Q^\beta}{n!} \langle \varphi_\alpha \psi^k, \tCoh{\psi_1},\ldots,\tCoh{\psi_n}, \left( \dfrac{\varphi_\gamma}{-z-\psi} \right) \rangle^X_{0,n+2,\beta} \cdot \varphi^\gamma . \end{equation*}
Thus the tangent space consists of vectors in $\Hcal$ of the form
\begin{equation*} \mathbf{r}(z) + \sum_{\beta,n} \dfrac{Q^\beta}{n!} \langle \mathbf{r}(\psi), \tCoh{\psi_1},\ldots,\tCoh{\psi_n}, \left( \dfrac{\varphi_\gamma}{-z-\psi} \right) \rangle^X_{0,n+2,\beta} \cdot \varphi^\gamma \end{equation*}
for $\mathbf{r}(z) \in \Hcal_+$. On the other hand, if we look at the expression \eqref{S* equation} given earlier for ${S_{\tCoh{z}}}^*(z) \in \operatorname{End}(\HH^*(X))\llbracket z^{-1} \rrbracket$, we see that this can be extended in a natural way to give a map $\Hcal_+ \to \Hcal$ via
\begin{equation*} {S_{\tCoh{z}}}^*(z)(\mathbf{r}(z)) = \mathbf{r}(z) + \sum_{\beta,n} \dfrac{Q^\beta}{n!} \langle \mathbf{r}(\psi) , \tCoh{\psi_1}, \ldots, \tCoh{\psi_n} , \left( \dfrac{\varphi_\gamma}{z-\psi} \right) \rangle^X_{0,n+2,\beta} \cdot \varphi^\gamma \end{equation*}
(note that this is \emph{different} from the extension of $S_{\tCoh{z}}(z)$ to an endomorphism of $\Hcal$ which we gave in \S \ref{subsection Fundamental solution matrix}, where we treated the insertion $\mathbf{r}(z)$ formally). Under the above definition, we see that:
\begin{equation*} \T_f \Lcal_X = {S_{\tCoh{z}}}^*(-z)(\Hcal_+).\end{equation*}
Fixing $\mathbf{r}(z), \mathbf{u}(z) \in \Hcal_+$, we thus need to show that
\begin{align*} \Omega \bigg( {S_{\tCoh{z}}}^*(z) (\mathbf{r}(-z)), {S_{\tCoh{z}}}^*(-z) (\mathbf{u}(z) ) \bigg)=0 \end{align*}
which is equivalent to:
\begin{align*} \operatorname{Res}_{z=0} \left( {S_{\tCoh{z}}}^*(z) \mathbf{r}(-z), {S_{\tCoh{z}}}^*(-z) (\mathbf{u}(z) ) \right) \diff\!z = 0.
\end{align*}
We take the moduli space
\begin{equation*} \overline{\mathcal{M}}_{0,n,(1),(1)}\left( (X\times\PP^1 \ | \ X_0 + X_\infty), (\beta,1) \right) \end{equation*}
as before and consider the equivariant integral (against the virtual class) of the following class:
\begin{equation*} \sum_{\beta,n} \dfrac{Q^\beta}{n!} \left( \ev_0^* ( \mathbf{r}(\psi_0) ) \cdot \prod_{i=1}^n \ev_i^* ( \tCoh{\psi_i} ) \cdot \ev_\infty^* (\mathbf{u}(\psi_\infty))\right). \end{equation*}
Then an analogous computation to the one given in \S \ref{section Proof of the proposition} shows that this integral is equal to:
\begin{equation*} \left( {S_{\tCoh{z}}}^*(z) (\mathbf{r}(-z)), {S_{\tCoh{z}}}^*(-z) (\mathbf{u}(z) ) \right).\end{equation*}
Thus the above pairing is a polynomial in $z$, and so in particular the coefficient of $z^{-1}$ vanishes. But this is precisely the residue that we needed to calculate, and the claim follows. \end{proof}

Another fundamental fact about $\Lcal_X$, already discussed in \S \ref{subsection Lagrangian cone}, is that:
\begin{equation*} (\T_f \Lcal_X) \cap \Lcal_X = z \cdot \T_f \Lcal_X. \end{equation*}
To finish, we will give a direct proof of one important consequence of this fact.
\begin{proposition} $f \in z \cdot \T_f\Lcal_X$.\end{proposition}
\begin{proof} As noted before, an immediate consequence of Proposition \ref{main result} is that:
\begin{equation*}S_{\tCoh{z}}(z)(f) \in z\cdot\Hcal_+ . \end{equation*}
Applying ${S_{\tCoh{z}}}^*(-z)$ to both sides, we find that
\begin{equation*} f \in {S_{\tCoh{z}}}^*(-z) \left( z \cdot \Hcal_+ \right) \end{equation*}
where, unlike in the proof of Proposition \ref{prop L a cone}, the extension of ${S_{\tCoh{z}}}^*(-z)$ from $\HH^*(X)$ to $\Hcal_+=\HH^*(X)[z]$ is obtained by expanding linearly in $z$. A deep fact from the theory now says that, under this definition:
\begin{equation*} {S_{\tCoh{z}}}^*(-z)(\Hcal_+) = \T_f \Lcal_X. \end{equation*}
Some care is required here: we also saw this statement in the proof of the previous proposition, but that was for a different extension of ${S_{\tCoh{z}}}^*(-z)$ which was not linear in $z$. Under the new extension used here, which is linear in $z$, the statement still holds, though it is much less trivial. Using this, we obtain
\begin{equation*} f \in z \cdot {S_{\tCoh{z}}}^*(-z)(\Hcal_+) = z \cdot \T_f \Lcal_X \end{equation*}
as required.
\end{proof}

\begin{remark} The idea of using torus localisation to prove that certain generating functions are polynomials is not new. It was used by Givental in the proof of the Mirror Theorem \cite{GiventalEquivariant} and by I. Ciocan-Fontanine and B. Kim in the proof of the wall-crossing formula for quasimap invariants \cite{CiocanFontanineKimWallMirror}. The disussion above constitutes a small continuation of this story. \end{remark}

\bibliographystyle{alpha}
\footnotesize{\bibliography{Bibliography}}

\bigskip\bigskip

\noindent Navid Nabijou \\
School of Mathematics and Statistics, University of Glasgow \\
\texttt{navid.nabijou@glasgow.ac.uk}

\end{document}